\documentclass{acm_proc_article-sp}

\usepackage{ifthen}

\newboolean{llncs}
\setboolean{llncs}{false}

\usepackage{amsfonts}

\usepackage{hyperref}

\usepackage{amssymb,amsmath}

\usepackage{graphicx}
\usepackage{algorithm}
\usepackage{algorithmic}

\newtheorem{prop}{Proposition}

\newtheorem{defn}{Definition}
\ifthenelse{\boolean{llncs}}{
}
{

}

\newcommand{\norm}[1]{\lVert #1 \rVert}

\newcommand{\abs}[1]{\left \lvert #1 \right \rvert}

\newcommand{\extr}{\mathrm{extr}}

\newcommand{\NEW}[1]{{\em #1}}

\ifthenelse{\boolean{llncs}}{
\author{Olivier Fercoq
\thanks{The doctoral work of the author is supported by Orange Labs through the research contract CRE~3795 with INRIA.}}
\title{PageRank optimization applied to spam detection}
\institute{INRIA Saclay and CMAP Ecole Polytechnique \\
\email{olivier.fercoq@inria.fr}}%
}
{}

\begin{document}
\ifthenelse{\boolean{llncs}}{}
{
\author{Olivier Fercoq
\titlenote{The doctoral work of the author is supported by Orange Labs through the research contract CRE~3795 with INRIA.}
\affaddr{INRIA Saclay and CMAP Ecole Polytechnique} \\
\email{olivier.fercoq@inria.fr}}%
\title{PageRank optimization applied to spam detection}
}

\maketitle

\begin{abstract}
We give a new link spam detection and PageRank demotion algorithm called MaxRank. Like TrustRank and AntiTrustRank, it starts with a seed of hand-picked trusted and spam pages. We define the MaxRank of a page as the frequency of visit of this page by a random surfer minimizing an average cost per time unit. On a given page, the random surfer selects a set of hyperlinks and clicks with uniform probability on any of these hyperlinks. The cost function penalizes spam pages and hyperlink removals. The goal is to determine a hyperlink deletion policy that minimizes this score. The MaxRank is interpreted as a modified PageRank vector, used to sort web pages instead of the usual PageRank vector. The bias vector of this ergodic control problem, which is unique up to an additive constant, is a measure of the ``spamicity'' of each page, used to detect spam pages. We give a scalable algorithm for MaxRank computation that allowed us to perform experimental results on the WEBSPAM-UK2007 dataset. We show that our algorithm outperforms both TrustRank and AntiTrustRank for spam and nonspam page detection.
\end{abstract}


\section{Introduction}

Internet search engines use a variety of algorithms to sort web pages based on their text
content or on the hyperlink structure of the web.
In this paper, we focus on algorithms that use the latter hyperlink structure, 
called link-based algorithms. The basic notion for all these algorithms is the web graph,
which is a digraph with a node for each web page and an arc between pages~$i$ 
and~$j$ if there is a~hyperlink from page~$i$ to page~$j$. 

One of the main link-based ranking methods is the Page\-Rank introduced by Brin and Page~\cite{Brin-Anatomy}.
It is defined as the invariant measure of a~walk made by a~random surfer
on the web graph. When reading a~given page, the surfer either
selects a~link from the current page (with a~uniform probability),
and moves to the page pointed by that link, or interrupts his current search,
and then moves to an~arbitrary page, which is selected
according to given ``teleportation'' probabilities.
The rank of a~page is defined as its frequency of visit
by the random surfer. It is interpreted as the ``popularity'' of the page.

From the early days of search engines, some webmasters
have tried to get their web pages overranked thanks to malicious manipulations.
For instance, adding many keywords
on a page is a classical way to make search engines consider a page relevant
to many queries.
With the advent of link-based algorithms, spammers have developed new strategies, called link-spamming~\cite{Gyongyi-taxonomy},
that intend to give some target page a high score.
For instance, Gy\"{o}ngyi and Garcia-Molina~\cite{linkspamalliances}
showed various linking strategies that improve the PageRank score of a page. They justified the presence of link farms with patterns with every page linking to
one single page.
Baeza-Yates, Castillo and L\'opez~\cite{collusionTopologies} also
showed that making collusions is a~good way to improve PageRank.

In order to fight such malicious manipulations that
 deteriorate search engines' results and deceive web surfers,
various techniques have been developed.
We refer to~\cite{CastilloDavison} for a detailed survey of this subject.
Each spam detection algorithm is focused on a particular
aspect of spam pages. Content analysis (see~\cite{Ntoulas-contentanalysis} for instance) is the main tool to
detect deceiving keywords. Some simple heuristics~\cite{AbouAssaleh-heuristicsSpamdexing}
may be enough to detect the most coarse link-spam techniques,
but more evolved graph algorithms like clique detection~\cite{Saito-clique},
SpamRank~\cite{Benczur-spamrank} or Truncated PageRank~\cite{Becchetti-linkspam}
have also been developed to fight link-spamming.
As web spammers adapt themselves to detection algorithms,
machine learning techniques~\cite{GanSuel-machineLearning} try to discover actual
predominant spam strategies and to adapt to its evolutions.
Another direction of research concerns the propagation of trust through the web graph
with the TrustRank algorithm~\cite{Gyongyi-Trustrank} and its variants
or the propagation of distrust through a reversed web graph with the AntiTrustRank algorithm~\cite{Krishnan-Antitrustrank}.

In this paper, we develop a new link spam detection and PageRank demotion algorithm called
MaxRank. Like in~\cite{Gyongyi-Trustrank, Krishnan-Antitrustrank}, we start with a seed of hand-picked trusted
and spam pages. 
We define the MaxRank of a page as the frequency of visit of this page by a random surfer 
minimizing an average cost per time unit. On a given page, the random surfer selects a 
set of hyperlinks and clicks with uniform probability on any of these hyperlinks. 
The cost function penalizes spam pages and hyperlink removals. The goal 
is to determine an optimal hyperlink deletion policy.
The features of MaxRank are based on PageRank optimization 
\cite{AvrLit-OptStrat, Viennot-2006, NinKer-PRopt,IshiiTempo-FragileDataPRcomp,Blondel-PRopt} and more
particularly on the results of \cite{Fercoq-PRopt}. 
Those papers have shown that the problem of optimizing the PageRank of a set of pages
by controlling some hyperlinks can be solved by Markov Decision process algorithms.
In these works, the optimization of PageRank was thought from a webmaster's point of view
whereas here, we take the search engine's point of view.
We show that the Markov Decision Process defining MaxRank is solvable in polynomial time
 (Proposition~\ref{prop:welldesc}),
because the polytopes
of occupation measures admit efficient (polynomial time) separation
oracles.
The invariant measure of the Markov Decision Process, the MaxRank vector, is interpreted as a modified PageRank vector, used to sort web pages
instead of the usual PageRank vector. The solution of the ergodic dynamic programming equation, 
called the bias vector, is unique up to 
an additive constant and is interpreted as a measure of the ``spamicity'' of each page, used to detect spam pages.

We give a scalable algorithm for MaxRank computation that allowed us
to perform numerical experiments on the WEBSPAM-UK2007 dataset \cite{webspam2007}. We show that
our algorithm outperforms both TrustRank and AntiTrustRank for spam and nonspam page detection.
As an example, on the WEBSPAM-UK2007 dataset, for a recall of 0.8 in the spam detection problem, MaxRank has a precision of
0.87 while TrustRank has a precision of 0.30 and AntiTrustRank a precision of 0.13.

\section{The PageRank and (Anti)- \\ TrustRank algorithms}

We first recall the basic elements of the Google PageRank computation, 
see~\cite{Brin-Anatomy} and~\cite{LanMey-Beyond} for more information.
We call \NEW{web graph} the directed graph with a~node per web page
and
an arc from page~$i$ to page~$j$ if page~$i$ contains a~hyperlink to page~$j$.
We identify the set of pages to $[n]:=\{1,\ldots,n\}$. 

Let $D_i$ denote the number of hyperlinks contained in page $i$.
Assume first that $D_i\geq 1$ for all $i\in [n]$, meaning that every page
has at least one outlink. Then, we construct
the $n\times n$ stochastic matrix $S$, which is such that
\begin{align}
S_{i,j}=\begin{cases}D_i^{-1} & \text{if page }j \text{ is pointed to from page }i\\
0 & \text{otherwise}
\end{cases}\label{e-def-S}
\end{align}

We also fix a~row vector $z \in \mathbb{R}_+^n $, the \NEW{zapping}
or \NEW{teleportation} vector, which must be stochastic (so, $\sum_{j\in[n]} z_j =1$), 
together with a~\NEW{damping factor} $\alpha \in [0,1]$ (typically $\alpha=0.85$) and define the new stochastic matrix
\begin{equation*}
 P=\alpha S + (1-\alpha) e z
\end{equation*}
where $e$ is the (column) vector in $\mathbb{R}^n$ with all entries equal to~$1$. 

When some page $i$ has no outlink, $D_i=0$, and so the entries of the $i$th row of the matrix $S$ cannot be defined according to~\eqref{e-def-S}.
Then, we set $S_{i,j}:=z_j$.

Consider now a~Markov chain
$(X_t)_{t\geq0}$ with transition matrix $P$, so that for all $i,j\in [n]$, 
$\mathbb{P}(X_{t+1}=j | X_t=i) = P_{i,j}$.
Then,
$X_t$ represents the position of a~websurfer at time $t$:
when at page $i$, the websurfer continues his current
exploration of the web with probability $\alpha$ and moves to the next
page by following the links included in page $i$, as above, or with
probability $1-\alpha$, stops his current exploration and then teleports
to page $j$ with probability~$z_j$.


The \NEW{PageRank} $\pi$ is defined as the invariant measure of 
the Markov chain $(X_t)_{t\geq 0}$ representing the behavior of the websurfer.
This invariant measure is unique if $\alpha<1$ or if $P$ is irreducible.


The TrustRank algorithm~\cite{Gyongyi-Trustrank} is a semi-automatic algorithm designed
for spam detection and PageRank demotion. It starts with a seed of
hand-picked trusted pages and then launches the PageRank algorithm with
a teleportation vector with nonzero entries only on this seed of trusted pages.
Thus the initial trust score of trusted pages will propagate through hyperlinks.
The fundamental idea is that trusted pages link to trusted pages and spam pages
are linked to by spam pages~\cite{Castillo-neighbors}.
Each coordinate of the vector obtained this way gives the TrustRank score of the associated page.
Pages with a high TrustRank score are trusted and considered to be non spam, pages with a small
Trustrank score are untrusted and considered to be spam. The boundary between trusted and
untrusted pages is an arbitrary threshold in TrustRank.
As it is a variant of PageRank, the TrustRank score can also be used as such to rank pages instead of the usual PageRank.
This is called PageRank demotion, where untrusted pages get a negative promotion of their score.

The AntiTrustRank algorithm~\cite{Krishnan-Antitrustrank} follows the same idea as TrustRank
but it uses a seed of spam pages instead of a seed of trusted pages and
launches the PageRank algorithm using reversed arcs, i.e., it considers
a reversed web graph where there is an arc between nodes $i$ and $j$ if page $j$ points to page~$i$.
Pages with a high AntiTrustRank are then considered to be spam pages.

\section{Well-described Markov decision processes}  \label{sec:wellDescMDP}

%

The definition of the MaxRank relies on the notion of Markov Decision Processes that we recall now.

A finite {\em Markov decision process} is a $4$-uple $(I,(A_i)_{i\in I},p,c)$
where $I$ is a finite set called the {\em state space}; for all $i \in I$, $A_i$ is the finite set of {\em admissible actions} in state $i$;
$p: I\times \cup_{i\in I} (\{i\} \times A_i) \to \mathbb{R}_+$ is the {\em transition law}, so that $p(j|i,a)$ is the probability to go to state $j$ form state $i$ when
action $a \in A_i$ is selected; and $c: \cup_{i\in I}( \{i\} \times A_i) \to \mathbb{R}$ is the {\em cost function}, so that $c(i,a)$ is the instantaneous cost
when action $a$ is selected in state $i$.

Let $X_t\in I$ denote the state of the system at the discrete time $t\geq 0$.
A {\em deterministic control strategy} $\nu$ is a~sequence of actions $(\nu_t)_{t\geq 0}$
such that for all $t \geq 0$, $\nu_t$ is a~function of the history $h_t=(X_0, \nu_0, \ldots, X_{t-1}, \nu_{t-1}, X_t)$ and $\nu_t \in A_{X_t}$.
Of course, $\mathbb{P}( X_{t+1}=j | X_t , \nu_t ) = p(j|X_t, \nu_t) , \forall j \in [n], \forall t \geq 0$.
More generally, we may consider {\em randomized} strategies $\nu$ where $\nu_t$ is a probability measure on $A_{X_t}$.
A strategy $\nu$ is {\em stationary} (feedback) if there exists a function $\bar{\nu}$
such that for all $t\geq 0$, $\nu_t(h_t) = \bar{\nu}(X_t)$.

Given an initial distribution $\mu$ representing the law of $X_0$, the average cost infinite horizon Markov decision problem, also called {\em ergodic control problem},  consists in maximizing 
\begin{equation}\label{eqn:ergo}
\liminf_{T \rightarrow +\infty} \frac{1}{T}
 \mathbb{E}(\sum_{t=0}^{T-1} c(X_t,  \nu_t))
\end{equation}
where the maximum is taken 
over the set of randomized control strategies $\nu$.
Indeed, the supremum is the same if it is taken only over the set 
of randomized (or even deterministic) stationary feedback strategies (Theorem~9.1.8 in~\cite{puterman-mdp} for instance). 

A Markov decision process is {\em unichain} if the transition matrix corresponding to every
stationary policy has a single recurrent class. Otherwise it is {\em multichain}. 
When the problem is unichain, its value does not depend on the initial distribution whereas
when it is not, one may consider a vector $(g_i)_{i \in I}$ where $g_i$ represents the value of the problem~\eqref{eqn:ergo} when starting from state $i$, meaning that the law of $X_0$ is the Dirac measure at point $i$.

We shall need an extension of the formalism of Markov Decision processes
in which the actions are implicitly described. This extension, developped in~\cite{Fercoq-PRopt},
 relies on the theory of Groetschel, Lov\'asz and Schrijver~\cite{Lovasz-geomAlgo}
of linear programming over polyhedra given by separation oracles.

\begin{defn}[Def.\ 6.2.2 of \cite{Lovasz-geomAlgo}]  \label{defn:welldescribed}
We say that a po\-lyhedron $\mathcal{B}$ has {\em facet-complexity} at most~$\phi$ if there exists
a~system of inequalities with rational coefficients that has solution set~$\mathcal{B}$ and
such that the encoding length of each inequality of the system (the sum of the number of bits of the rational numbers appearing as coefficients
in this inequality) is at most~$\phi$. 

A {\em well-described polyhedron} is a~triple $(\mathcal{B};n,\phi)$ where $\mathcal{B} \subseteq \mathbb{R}^n$
is a~polyhedron with facet-complexity at most $\phi$. The {\em encoding length} of $\mathcal{B}$ is by definition $n+\phi$.
\end{defn}
\begin{defn}[Prob. (2.1.4) of \cite{Lovasz-geomAlgo}]
A~{\em strong separation oracle} for a set $K$ is an~algorithm that solves the
following problem: given a~vector~$y$, 
decide whether $y \in K$ or not and if not, find a~hyperplane
that separates $y$ from~$K$; 
i.e.,
find a~vector~$u$ such
that $u^T y > \max \{u^T x, x\in K\}$.
%
\end{defn}

Inspired by Definition~\ref{defn:welldescribed}, we introduced the following notion.
\begin{defn}[\cite{Fercoq-PRopt}] \label{defn:welldescribedMDP}
A finite Markov decision process $(I, (A_i)_{i\in I}, p, c)$ is {\em well-described} 
if for every state $i \in I$, we have $A_i \subset \mathbb{R}^{L_i}$
for some $L_i \in \mathbb{N}$, 
if there exists $\phi \in \mathbb{N}$
such that  the convex hull of every action set $A_i$ is a~well-described polyhedron $(\mathcal{B}_i; L_i, \phi)$
with a~polynomial time strong separation oracle, and if
the costs and transition probabilities satisfy
$c(i,a)=\sum_{l \in [L_i]} a_l C_{i}^l$
and
$p( j | i , a) = \sum_{l \in [L_i]} a_l Q_{i,j}^l$, $\forall i,j \in I$, 
$\forall a\in  A_i$,
where $C_{i}^l$ and $Q_{i,j}^l$
are given rational numbers,
for $i,j\in I$ and $l\in [L_i]$.

The {\em encoding length} of a well-described Markov decision process is by definition the sum of the encoding lengths of the rational numbers $Q_{i,j}^l$ and $C_{i}^l$ and of the well-described polyhedra~$\mathcal{B}_i$.
\end{defn}

The interest of well described Markov decision processes is that,
even if they have a number of actions exponential in the size of the problem,
Theorem~3 in~\cite{Fercoq-PRopt} shows that
the average cost infinite horizon problem for a well-described (multichain) Markov decision process can be solved in a time
polynomial in the input length.

%


\section{The MaxRank algorithm} 

In this section, we define the MaxRank algorithm.
It is based on our earlier works on PageRank optimization~\cite{Fercoq-PRopt}.
In~\cite{Fercoq-PRopt}, we considered the problem of optimizing the PageRank
of a given website from a webmaster's point of view, that is with
some controlled hyperlinks and design constraints.
Here, we take the search engine's point of view.
Hence, for every hyperlink of the web, we can choose to take it
into account or not: 
our goal is to forget spam links while letting trusted links active in the determination of the ranking.

As in TrustRank~\cite{Gyongyi-Trustrank} and AntiTrustRank~\cite{Krishnan-Antitrustrank}, we start with a seed of
trusted pages and known spam pages. The basic idea is to minimize the sum of PageRank scores
of spam pages and maximize the sum of PageRank scores of nonspam pages,
by allowing us to remove some hyperlinks when computing the PageRank.
However, if we do so, the optimal strategy simply consists in isolating
spam pages from trusted pages: there is then no hope to detect other
spam and nonspam pages. Thus, we add a penalty when a hyperlink is
removed, so that ``spamicity'' can still propagate through (removed or not removed) hyperlinks.
Finally, we also control the teleportation vector in order to
penalize further pages that we suspect to be spam pages.

We model this by a controlled random walk on the web graph, in which
the hyperlinks can be removed. Each time the random surfer goes to a spam page,
he gets a positive cost, each time he goes to a trusted page,
he gets a negative cost. When the status of the page is unknown, no cost incurs.
In addition to this a priori cost, he gets a penalty for each hyperlink removed.
Like for PageRank, the random surfer teleports with probability $\alpha$
at every time step; however, in this framework, he chooses the set of pages
to which he wants to teleport.

Let $\mathcal{F}_x$ be
the set of pages pointed by $x$ in the original graph and $D_x$ be
the degree of $x$. An action consists in determining $J \subseteq \mathcal{F}_x$,
the set of hyperlinks that remain,
and $I \subseteq [n]$, the set of pages to which the surfer may teleport.
We shall restrict $I$ to have a cardinality equal to $N \leq n$.
Then, the probability of transition from page $x$ to page $y$
is
\[
 p(y|x, I, J) = \alpha \nu_y(I,J) +(1-\alpha) z_y(I)
\]
where the teleportation vector and the hyperlink click probability distribution are given by
\begin{align*}
z_y(I)=\begin{cases}
\abs{I}^{-1} & \text{if }y \in I \\
0 & \text{otherwise}
\end{cases} 
\end{align*}
\begin{align*}
\nu_{y}(I,J)=\begin{cases}
z_y(I) & \text{if } \abs{J}=\emptyset \\
\abs{J}^{-1} & \text{if }y \in J \\
0 & \text{otherwise}
\end{cases}
\end{align*}
The cost at page $x$ is given by
\[
c(x,I,J)=c'_x + \gamma \frac{D_x-\abs{J}}{D_x} \enspace. 
\]
$c'_x$ is the a~priori cost of page $x$.
This a~priori cost should be positive for known spam pages and negative for trusted pages. 
$D_x$ is the degree of $x$ in the web graph and $\gamma>0$ is a penalty factor.
The penalty $\gamma \frac{D_x-\abs{J}}{D_x}$ is proportional to the number of pages removed.

We study the following ergodic control problem:
\begin{align} \label{eqn:ergodicPbMaxrank}
\inf_{(I_t)_{t\geq 0}, (J_t)_{t\geq 0}} \limsup_{T \rightarrow +\infty} \frac{1}{T}
 \mathbb{E}(\sum_{t=0}^{T-1} c(X_t,I_t,J_t)) \enspace ,
\end{align}
where an admissible control consists in selecting,
at each time step $t$, a subset of pages $I_t \subseteq [n]$ with $\abs{I_t}=N$
to which teleportation
is permitted, and a subset $J_t \subseteq \mathcal{F}_{X_t}$ of the set of hyperlinks
in the currently visited page $X_t$.

The following proposition gives an alternative formulation
of Problem~\eqref{eqn:ergodicPbMaxrank} that we will then show to be well-described.

\begin{prop} \label{prop:rewriting}
Fix $N \in \mathbb{N}$ and let
\[
 Z=\{ z \in \mathbb{R}^n | \sum_{i \in [n]} z_i =1 , 0 \leq z_i \leq \frac{1}{N} \} \enspace .
\] 
Let $\mathcal{F}_x$ be
the set of pages pointed by $x$ in the original graph and $D_x$ be
the degree of $x$.
Let $\mathcal{P}_x$ be the polyhedron defined as the set of vectors
 $(\sigma, \nu) \in \mathbb{R}^{D_x+1} \times \mathbb{R}^{n}$
such that there exists $w \in \mathbb{R}^{(D_x+1) \times n}$ verifying
\begin{subequations}
\begin{align}
& \sum_{d=0}^{D_x} \sigma^d =1 & &  \label{eqa} \\ 
&\sigma^d\geq 0 \; , & &
\forall d \in \{ 0, \ldots, D_x \} \label{eqb} \\
&\nu_j = \sum_{d=0}^{D_x} w^d_j \; ,& &\forall j \in [n] \label{eqc} \\
&\sum_{j \in [n]} w_j^d = \sigma^d \; ,& & \forall d \in \{ 0, \ldots, D_x \} \label{eqd} \\
&0 \leq w_j^0 \leq \frac{\sigma^0}{N_z} \; ,& & 
\forall j \in [n] \label{eqe} \\
&w_j^d = 0\; ,& &  \forall j \not \in \mathcal{F}_x, \forall d \in \{ 1, \ldots, D_x \} \label{eqf} \\
&0 \leq w_j^d \leq \frac{\sigma^d}{d} \; ,& & \forall j \in \mathcal{F}_x, \forall d \in \{ 1, \ldots, D_x \} \label{eqg} 
\end{align}
\end{subequations}
Then Problem \ref{eqn:ergodicPbMaxrank} is
equivalent to the following ergodic control problem
\begin{align} \label{eqn:ergodicPbMaxrankWellDescribed}
\inf_{\sigma, \nu, z} \limsup_{T \rightarrow +\infty} \frac{1}{T}
 \mathbb{E}(\sum_{t=0}^{T-1} \tilde{c}(X_t,\sigma_t, \nu_t, z_t)) \enspace ,
\end{align}
where the cost is defined as 
\[
\tilde{c}(x,\sigma, \nu,z)=c'_x + \gamma \frac{D_x-\sum_{d=0}^{D_x} d \sigma^d }{D_x} 
\]
and the transitions are 
\[
\tilde{p}(y|x,\sigma, \nu, z) = \alpha \nu_y + (1-\alpha) z_y \enspace .
\]
The admissible controls verify for all $t$, $(\sigma_t, \nu_t) \in \extr(\mathcal{P}_{X_t})$ (the set of extreme point of the polytope) and $z_t \in \extr(Z)$.

Indeed, to each action $(\sigma, \nu, z)$ of Problem~\eqref{eqn:ergodicPbMaxrankWellDescribed} corresponds a unique 
action $I$, $J$ of Problem~\eqref{eqn:ergodicPbMaxrank} and vice versa. Moreover, the respective transitions and costs are equal.
\end{prop}
\begin{proof}
Fix a page $x$ in $[n]$.
The extreme points of $Z$ are the vectors of $\mathbb{R}^n$ with $N$ coordinates equal to
$\frac{1}{N}$ and the other ones equal to 0. Hence $z(I) \in \extr(Z)$
and for each extreme point $z'$ of $Z$ there exists $I \in [n]$ such that
$\abs{I}=N$ and $z'=z(I)$.
We shall also describe the set of extreme points of $\mathcal{P}_x$.

From the theory of disjunctive linear programming~\cite{Balas-disjunctiveOpt},
we can see that the polytope $\mathcal{K}=\{ \nu \; | \; (\sigma, \nu) \in \mathcal{P}_x \}$ 
is the convex hull of the union of $D_x+1$ polytopes that we will denote $K_d$,
$d \in \{0, 1 \ldots D_x \}$. If $d=0$, then $K_0=Z$.
If $d>0$, $K_d = \{\nu \in \mathbb{R}^n \; | \; \sum_{j \in [n]} \nu_j=1, \; 0 \leq \nu_j \leq \frac{1}{d}, \forall j \in \mathcal{F}_x,\; \nu_j =0 , \forall j \not \in \mathcal{F}_x \}$.

Let $(\sigma, \nu)$
be an extreme point of $\mathcal{P}_x$. By Corollary~2.1.2-ii) in~\cite{Balas-disjunctiveOpt},
 there exists $d^*$ such that
$\sigma^{d^*}=1$ and $\nu$ is an extreme point of $\mathcal{K}$.
As $\sigma^{d^*}=1$ and $\sigma_d=0$ for $d \not = d^*$, we conclude that $\nu$ is also an extreme point of $K_{d^*}= \mathcal{P}_x \cap \{ \sigma | \sigma^{d^*}=1 \}$. 
If $d^*=0$, $\nu \in \extr(Z)$ and if $d^*>0$, the extreme points of $K_{d^*}$ correspond exactly to the vectors with 
$d^*$ coordinates in $\mathcal{F}_x$ equal to
$\frac{1}{d^*}$ and the other ones equal to 0.
Hence there exists $I \subseteq [n]$ and $J \subseteq \mathcal{F}_x$
such that $\abs{I}=N$, $\abs{J}=d^*$ and $\nu=\nu(I,J)$.

Conversely, fix $I$ and $J$ and let $\sigma(J)$ be such that $\sigma^d(J)=1$ if and only if $d=\abs{J}$.
Then $(\sigma(J), \nu(I,J))$ is an extreme point of $\mathcal{P}_x$ by Corollary~2.1.2-i) in~\cite{Balas-disjunctiveOpt}.

Finally the costs are the same since $\abs{J}=\sum_{d=0}^{D_x} d \sigma^d$.
\end{proof}

%

\begin{prop} \label{prop:welldesc}
If $\alpha$, $\gamma$ and $c'_i$, $i \in [n]$ are rational numbers, then 
the ergodic control problem~\eqref{eqn:ergodicPbMaxrankWellDescribed} is the average cost infinite horizon problem for a well described Markov decision process and it is polynomial time solvable.
\end{prop}
\begin{proof}
 Clearly, the process described is a Markov decision process. 
As the polytopes $\mathcal{P}_i$, $i \in [n]$ and $Z$ are described by
a polynomial number of inequalities with at most $n+1$ terms in each, 
they are
well described. Indeed, the separation oracle consisting simply in testing each inequality terminates in polynomial time.
The cost and transitions are linear functions on those polytopes with rational coefficients since $\alpha$ and $c'_i$, $i \in [n]$ are rational numbers.
Thus the Markov decision process is well described. By Theorem~3 in~\cite{Fercoq-PRopt}, Problem~\eqref{eqn:ergodicPbMaxrankWellDescribed}
is thus solvable in polynomial time.
\end{proof}

\begin{prop} \label{prop:dynprog}
The dynamic programming equation
\begin{multline} \label{eqn:progdyn}
v_i + \lambda = \min_{ (\sigma, \nu) \in \mathcal{P}_{i}, z \in Z} 
c'_i + \gamma \frac{D_i-\sum_{d=0}^{D_i} d \sigma^d }{D_i} \\+ 
\sum_{j \in [n]} (\alpha \nu_j + (1-\alpha) z_j) v_j \; , \quad \forall i \in [n]
\end{multline}
has a~solution $v \in \mathbb{R}^n$ and $\lambda \in \mathbb{R}$. 
The constant $\lambda$ is unique and is the value of problem~\eqref{eqn:ergodicPbMaxrank}. 
An optimal strategy is obtained by selecting
for each state~$i$, $(\sigma,\nu) \in \mathcal{P}_i$ and $z \in Z$ maximizing Equation~\eqref{eqn:progdyn}.
The function $v$ is called the \NEW{bias}.
\end{prop}
\begin{proof}
Theorem~8.4.3 in~[Put94] applied to the unichain ergodic control
problem~\eqref{eqn:ergodicPbMaxrankWellDescribed} implies the result of the proposition but with $\mathcal{P}_i$
replaced by $\extr(\mathcal{P}_i)$. But as the expression which is maximized is affine,
using $\mathcal{P}_i$ or $\extr(\mathcal{P}_i)$ yields the same solution.
Proposition~\ref{prop:rewriting} gives the equivalence between \eqref{eqn:ergodicPbMaxrank} and \eqref{eqn:ergodicPbMaxrankWellDescribed}
\end{proof}


\begin{prop}  \label{prop:contrrate}
 Let $T$ be the dynamic programming operator $\mathbb{R}^n\to \mathbb{R}^n$ defined by
\begin{equation*}
T_i(v) = \min_{ (\sigma, \nu) \in \mathcal{P}_{i}} 
c'_i + \gamma \frac{D_i-\sum_{d=0}^{D_i} d \sigma^d }{D_i} + 
\alpha \sum_{j \in [n]} \nu_j v_j \; , \; \forall i \in [n] .
\end{equation*}
The map $T$ is $\alpha$-contracting in the sup norm and its fixed point $v$, which is unique,  is such that
$(v, (1-\alpha) \min_{z \in Z} z \cdot v)$ is solution of the ergodic dynamic programming equation~\eqref{eqn:progdyn}.
\end{prop}
\begin{proof}
 The set $\{\nu \text{ st: } (\sigma, \nu) \in \mathcal{P}_i\}$ is a~set of probability measures so $\lambda \in \mathbb{R} \Rightarrow T(v+\lambda) = T(v)+\alpha \lambda$ and $v \geq w \Rightarrow T(v) \geq T(w)$.
This implies that $T$ is $\alpha$-contracting. Let $v$ be its fixed point. For all $i \in [n]$, 
\begin{multline*}
v_i + (1-\alpha) \min_{z \in Z} z \cdot v \\
= \!\!\! \min_{ (\sigma, \nu) \in \mathcal{P}_{i}, z \in Z} \!\!
c'_i + \gamma \frac{D_i-\sum_{d=0}^{D_i} d \sigma^d }{D_i} + 
\sum_{j \in [n]}\! \alpha \nu_j v_j + (1-\alpha) z_j v_j
\end{multline*}
We get equation \eqref{eqn:progdyn} with constant $(1-\alpha) \min_{z\in Z} z \cdot v$.
\end{proof}

We can then solve the dynamic programming equation~\eqref{eqn:progdyn} and so the
ergodic control problem~\eqref{eqn:ergodicPbMaxrank} 
by value iteration (outer loop of Algorithm~\ref{alg:maxrank}).

The algorithm starts with an~initial potential function $v$, 
scans repeatedly the pages and updates $v_i$ when $i$ is the current
page according to $v_i \leftarrow T_i(v)$
until convergence is reached.
Then $(v , (1-\alpha) \min_{z \in Z} z v)$ is solution of the ergodic dynamic programming equation~\eqref{eqn:progdyn} and an optimal linkage strategy is recovered by selecting the
maximizing $(\sigma, \nu)$ at each page. An optimal teleportation vector is recovered by selecting
a maximizing $z$ in $\min_{z \in Z} z v$.

Thanks to the damping factor $\alpha$, the iteration can be seen to be $\alpha$-contracting if the pages are scanned in a~cyclic order.
Thus the algorithm converges in a~number of steps independent of the
dimension of the web graph.

For the evaluation of the dynamic programming operator $T$ at a page $i$
(inner for-loop of Algorithm~\ref{alg:maxrank}),
 we remark that for a fixed $0-1$ valued $\sigma$,
that is for a fixed number of removed hyperlinks, the operator $T_i$ 
maximizes a linear function on a hypercube,
which reduces essentially to a sort.
Any extreme point $(\sigma, \nu)$ of $\mathcal{P}_i$ necessarily
verifies that $\sigma$ is  $0-1$ valued. Thus we just need to choose the best value of $\sigma$ among the $D_i +1$ possibilities.

\begin{algorithm}
\caption{MaxRank algorithm}
\label{alg:maxrank}
 \begin{algorithmic}[1]
\STATE  Initialization: $v \in \mathbb{R}^n$
\WHILE{$\norm{v-T(v)}_{\infty} \geq \epsilon$}
\STATE Sort $(v_l)_{l \in [n]}$ in increasing order and let $\phi : [n]  \rightarrow [n]$ be the sort  function so that 
$v_{\phi(1)} \leq \dots\leq v_{\phi(n)}$.
\STATE $\lambda \leftarrow \frac{1-\alpha}{N} \sum_{j=1}^N v_{\phi(j)}$
\FOR{$i$ from 1 to $n$}
\STATE $w_i^0 \leftarrow c_i' +\gamma +\frac{\alpha}{1-\alpha}\lambda$
\STATE Sort $(v_j)_{j \in \mathcal{F}_i}$ in increasing order and let $\psi : \mathcal{F}_i  \rightarrow \{1, \ldots, \abs{\mathcal{F}_i} \}$ such that 
$v_{\psi(1)} \leq \dots \leq v_{\psi(\abs{\mathcal{F}_i})}$.
\FOR{$d$ from 1 to $D_i$}
\STATE $w_i^d \leftarrow c_i' + \gamma \frac{D_i-d}{D_i} + \frac{\alpha}{d} \sum_{j=1}^{d} v_{\psi(j)}$
\ENDFOR
\STATE $v_i \leftarrow T_i(v) = \min_{d \in \{0, 1, \ldots, D_i \}} w_i^d$
\ENDFOR
\ENDWHILE
 \end{algorithmic}
\end{algorithm}

This very efficient algorithm is highly scalable:  we used it for our experimental results
on a large size dataset (Section~\ref{sec:num}).

The following proposition shows that if $\gamma$ is too big, then the optimal
link removal strategy is trivial. It also gives an interpretation of the bias vector
in terms of number of visits of spam pages.

\begin{prop}
If $\gamma > \frac{2\alpha}{1-\alpha}\norm{c}_{\infty}$, then no link should be removed.
Moreover if in addition, $c'_i=1$ when $i$ is a spam page and 0 otherwise, then
the $i$th coordinate of the fix point of operator $T$ in Proposition~\ref{prop:contrrate}
is equal to the expected mean number of spam pages visited before teleportation when starting
a walk from page $i$.
\end{prop}
\begin{proof}
 Proposition~\ref{prop:contrrate} gives a normalization of the bias vector
such that it cannot take any value bigger than $\frac{\norm{c}_{\infty}}{1-\alpha}$,
as it is the case for PageRank optimization~\cite{NinKer-PRopt, Fercoq-PRopt}.

Now fix $i \in [n]$. Let $\nu^0=\nu(\mathcal{F}_i)$ and $\nu=\nu(I,J)$ for $J \subset \mathcal{F}_i$.
If $J= \emptyset$, then 
\[
\alpha \abs{\nu v - \nu^0 v} \leq \alpha \abs{\nu v}+\alpha \abs{\nu^0 v}
\leq  \frac{2\alpha}{1-\alpha}\norm{c}_{\infty} \frac{D_i-\abs{\emptyset}}{D_i} \enspace .
\]
If $J \not= \emptyset$, then
\begin{multline*}
 \alpha \abs{\nu v - \nu^0 v} \leq \alpha \abs{(\frac{1}{\abs{J}}-\frac{1}{D_i})\sum_{j \in J} v_j} 
+ \alpha \abs{\frac{1}{D_i}\sum_{j \in \mathcal{F}_i} v_j} \\
\leq \frac{D_i-\abs{J}}{\abs{J} D_i}
\frac{\alpha \norm{c}_{\infty}}{1-\alpha} + \frac{D_i-\abs{J}}{D_i}\frac{\alpha \norm{c}_{\infty}}{1-\alpha} \leq \gamma \frac{D_i-\abs{J}}{D_i}.
\end{multline*}
This proves that choosing $J=\mathcal{F}_i$ is always the best strategy when
 $\gamma > \frac{2\alpha}{1-\alpha}\norm{c}_{\infty}$. 

When no link is removed and $c'$ is defined as in the proposition,
we are in the framework of~\cite{NinKer-PRopt}, where it is shown that
the $i$th coordinate of the fix points of the operator $T$ is equal to the expected 
mean number of visits before teleportation when starting
a walk from page $i$.
\end{proof}

\section{Spam detection and PageRank\\ demotion} \label{sec:num}

We performed numerical experiments on the WEBSPAM-UK2007 dataset \cite{webspam2007}.
This is a crawl of the .uk websites with $n=105,\!896,\!555$ pages performed in 2007,
associated with lists of hosts classified as spam, nonspam and borderline. There is
a training dataset for the setting of the algorithm and a test dataset
to test the performance of the algorithm.

We took $\gamma=4$, $\alpha=0.85$, $N=0.89 n$, $c'_i=1$ if $i$ is a spam page of the training dataset,
 $c'_i=-0.2$ if $i$ is a nonspam page of the training dataset and $c'_i=0$ otherwise.
Then we obtained a MaxRank score and the associated bias vector.
We also computed TrustRank and AntiTrustRank with the training dataset as the seed sets.
We used the Webgraph framework~\cite{BoldiVigna-webgraph}, so that we could
manage the computation on a personal computer with four Intel Xeon CPUs at 2.98~Ghz and 8~GB RAM.
We coded the algorithm in a parallel fashion thanks to the OpenMP library. Computation took around 6 minutes for each evaluation of the
dynamic programming operator of Proposition~\ref{prop:contrrate}
and 6 hours for 60 such iterations (precision on the objective $\alpha^{60} \leq 6.10^{-5}$).
By comparison, PageRank computation with the same precision required 1.3 hour on the same computer,
which is of the same order of magnitude.

Figure~\ref{fig:maxrankvalues} gives the values taken by the bias vector.
 Figure~\ref{fig:spamdetection} compares the precision and recall
 of PageRank, TrustRank, AntiTrustRank and MaxRank bias for spam or non spam detection.
Precision and recall are the usual measures of the quality 
of an information retrieval algorithm~\cite{BaezaYates-modernIR}.
Precision is the probability that a randomly selected retrieved document is relevant.
Recall is the probability that a randomly selected relevant document is retrieved in a search.
These values were obtained using the training and the test sets.
Figure~\ref{fig:spamdemotion} compares TrustRank and MaxRank scores for PageRank demotion.

\begin{figure}
 \centering
\includegraphics[width=26em]{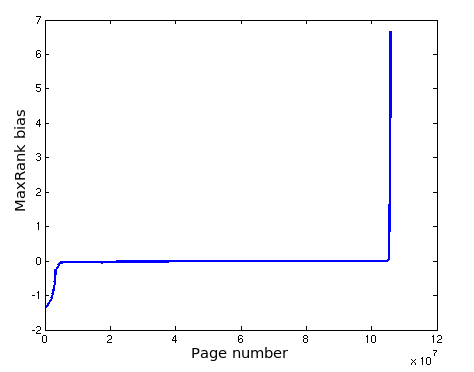}
\includegraphics[width=26em]{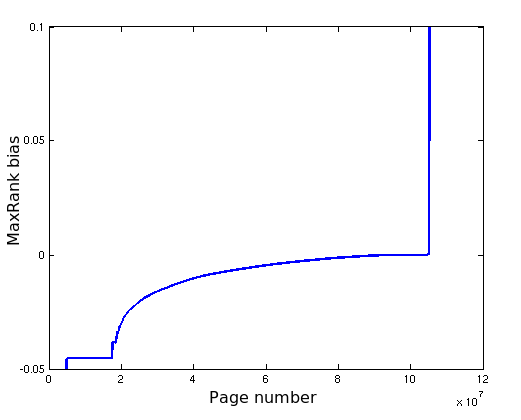}
\ifthenelse{\boolean{llncs}}{
\caption{Recognizing spam with the bias vector. Left:  the values of the bias for all pages. Right: zoom on bias values near 0. Pages are sorted by growing bias value. Spam pages of the training set
have a large positive bias value, non spam pages of the training set have a negative bias value.
Pages in between describe a ``continuum'' of values and the separation between pages considered spam
or not is arbitrarily set.} 
}
{
\caption{Recognizing spam with the bias vector. Top: the values of the bias for all pages. Bottom: zoom on bias values near 0. Pages are sorted by growing bias value. Spam pages of the training set
have a large positive bias value, non spam pages of the training set have a negative bias value.
Pages in between describe a ``continuum'' of values and the separation between pages considered spam
or not is arbitrarily set.} 
}
\label{fig:maxrankvalues}
\end{figure}
\begin{figure} 
 \centering
\includegraphics[width=26em]{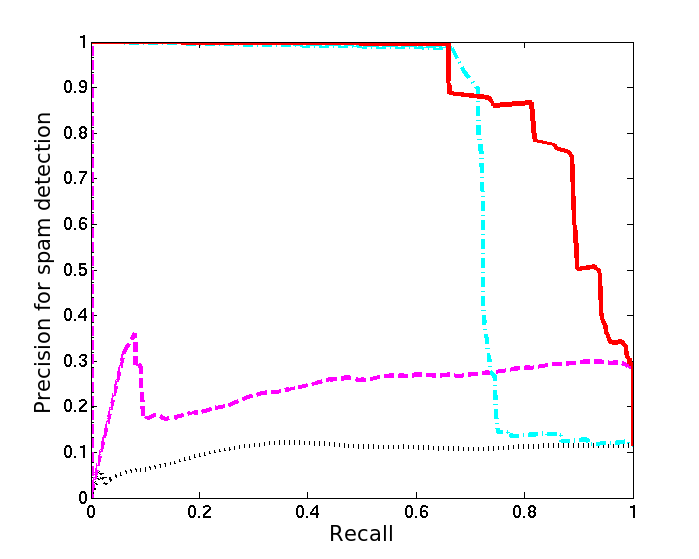}
\includegraphics[width=26em]{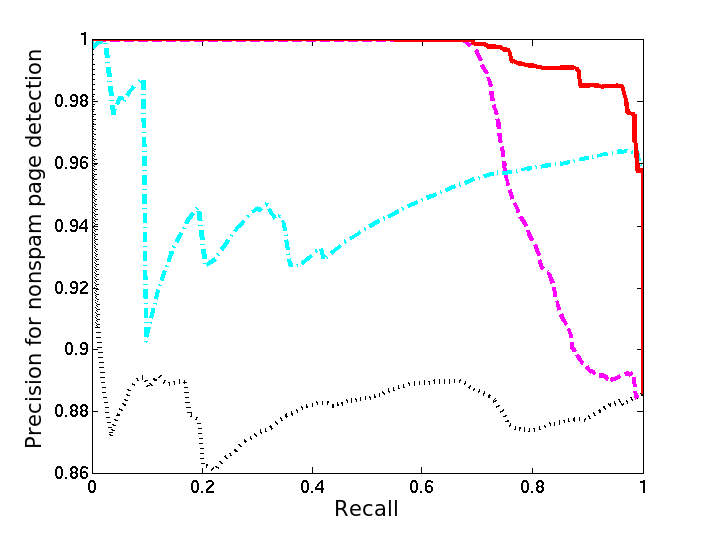}
\ifthenelse{\boolean{llncs}}
{
\caption{Left: Precision as a function of recall for spam detection.
Right: Precision as a function of recall for detection of non spam pages.
We present the result for four algorithms: 
PageRank (dotted line below), TrustRank (dashed line),
AntiTrustRank (dash-dotted line) and MaxRank bias (solid line). Given a vector of scores, the various pairs of precision-recall values
are obtained for various thresholds defining the discrimination between pages considered spam and non spam. 
TrustRank and AntiTrustRank have a precision of 1 on their seed set, which represent around 70\%
of the total test set. But out of their training set, their precision decreases quickly. 
MaxRank, on the other hand, remains precise even out of its training set.}
}
{
\caption{Top: Precision as a function of recall for spam detection.
Bottom: Precision as a function of recall for detection of non spam pages.
We present the result for four algorithms: 
PageRank (dotted line below), TrustRank (dashed line),
AntiTrustRank (dash-dotted line) and MaxRank bias (solid line). Given a vector of scores, the various pairs of precision-recall values
are obtained for various thresholds defining the discrimination between pages considered spam and non spam. 
TrustRank and AntiTrustRank have a precision of 1 on their seed set, which represent around 70\%
of the total test set. But out of their training set, their precision decreases quickly. 
MaxRank, on the other hand, remains precise even out of its training set.}
}

\label{fig:spamdetection}
\end{figure}

\begin{figure} 
 \centering
\includegraphics[width=26em]{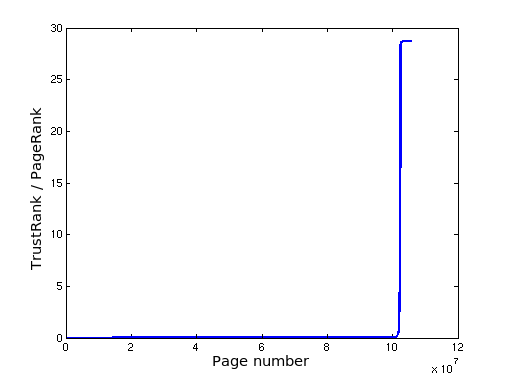}
\includegraphics[width=26em]{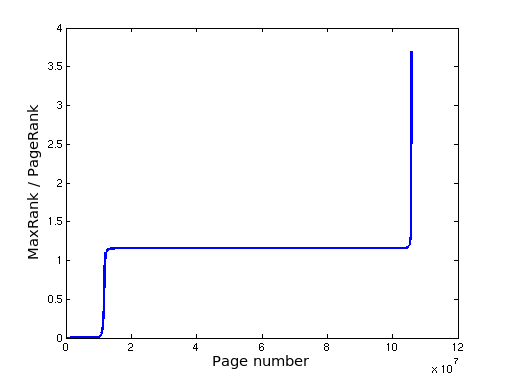}
\ifthenelse{\boolean{llncs}}
{
\caption{Left: PageRank demotion by TrustRank. We show the ratio of TrustRank over PageRank. Pages are sorted by growing promotion.
The demotion is very coarse, since there are many pages with a very small TrustRank score, when compared to their PageRank score. Right: PageRank demotion by MaxRank. The demotion is finer, most of the pages have a nearly unchanged score, up to a multiplicative constant.}
}
{
\caption{Top: PageRank demotion by TrustRank. We show the ratio of TrustRank over PageRank. Pages are sorted by growing promotion.
The demotion is very coarse, since there are many pages with a very small TrustRank score, when compared to their PageRank score. Bottom: PageRank demotion by MaxRank. The demotion is finer, most of the pages have a nearly unchanged score, up to a multiplicative constant.}
}
\label{fig:spamdemotion}
\end{figure}


\ifthenelse{\boolean{llncs}}{
\bibliographystyle{splncs03.bst}
}
{
\bibliographystyle{abbrv}
}
\bibliography{linkspam,../PAPEROPTIMRANK/pagerank}

\end{document}